\documentclass[a4paper,12pt]{article}
\usepackage[utf8]{inputenc}
\usepackage{amsmath,amsthm,amssymb}

\usepackage{currfile}
\usepackage{navigator}
\IfFileExists{\currfilename}{\embeddedfile{sourcefile}{\currfilename}}{}
  
\title{Equitable $[[2,10],[6,6]]$-partitions \\of the $12$-cube%
\thanks{A part of the work was funded
by the Russian Science Foundation (grant 18-11-00136);
the other part was supported within the framework 
of the state contract of the Sobolev Institute
of Mathematics (FWNF-2022-0017).
}
}
\author{Denis S. Krotov%
\thanks{Sobolev Institute of Mathematics, Novosibirsk, Russia.
E-mail: \texttt{krotov@math.nsc.ru}}
}
\date{}

\newtheorem{lemma}{Lemma}
\newtheorem{theorem}{Theorem}
\theoremstyle{remark}
\newtheorem{remark}{Remark}
\newtheorem{problem}{Problem}

\begin{document}

\maketitle

\begin{abstract}
We describe the computer-aided classification of equitable partitions of the $12$-cube with quotient matrix $[[2,10],[6,6]]$, or, equivalently, simple orthogonal arrays OA$(1536,12,2,7)$, or order-$7$ cor\-re\-la\-tion-immune Boolean functions in $12$ variables with $1536$ ones (which completes the classification of unbalanced order-$7$ correlation-immune Boolean functions in $12$ variables). We find that there are $103$ equivalence classes of the considered objects, and there are only two almost-OA$(1536,12,2,8)$ among them. Additionally, we find that there are $40$ equivalence classes of pairs of disjoint simple OA$(1536,12,2,7)$ (equivalently, equitable partitions of the $12$-cube with quotient matrix $[[2,6,4], [6,2,4], [6,6,0]]$) and discuss the existence of a non-simple \linebreak[5] OA$(1536,12,2,7)$.

\end{abstract}

\section{Introduction}

The main result of the current work is 
completing the classification 
of unbalanced order-$7$ correlation-immune Boolean functions in $12$ variables
(equivalently, equitable $2$-partitions of the $12$-cube with asymmetric quotient matrix having an eigenvalue $-4$).
We solve the case of functions with $1536$ ones, finding $103$ equivalence classes,
while the functions with $1024$ ones and with $1792$ ones from the considered family
were characterized in~\cite{KroVor:2020} 
($16$ equivalence classes and $2$ equivalence classes, respectively).
Computationally, the classification considered in the current paper
is harder than the one in~\cite{KroVor:2020}; 
it requires the subdivision of the classified
family into three subfamilies, each of which is solved with a special
modification of the algorithm, and still spends almost thirty CPU years.
The classification~\cite{KroVor:2020} of functions with $1024$ ones
was much more easy and straightforward because of smaller number of ones,
while the $1792$-ones case was solved using the Fourier analysis, 
which required more analytical than computational work
(the crucial fact for the success of the Fourier-based classification was 
that the average density $1792/2^{12}=9/16$ of ones of the function was not a multiple
of $1/8$, which is not true for the  $1536$-ones case).

The \emph{$n$-cube} is a graph $Q_n$ on the set $\{0,1\}^n$ of binary $n$-tuples, 
where two  $n$-tuples are adjacent if and only if they differ in exactly one 
position.
A partition $(C_0,\ldots,C_{k-1})$ 
of the vertex set of $Q_n$
into two \emph{cells}
is called an 
\emph{equitable $k$-partition} (of $Q_n$)
with {\it quotient} $k \times k$ matrix $S=(S_{ij})$ 
if for all $i,j \in \{0,\ldots,k-1\}$ any vertex of~$C_i$ 
has exactly $S_{ij}$ neighbors in~$C_j$.
A (Boolean) function $f: \{0,1\}^n \to \{0,1\}$ is called \emph{unbalanced} 
if the number of its ones is different from~$0$, $2^{n-1}$, and~$2^n$.
It is called \emph{$t$-th order correlation immune} if the number of ones (equivalently, zeros)
$(x_1,\ldots,x_n): f (x_1,\ldots,x_n)=1 $
is statistically independent on the values of any $t$ arguments;
that is, it is the same in all subgraphs of~$Q_n$ isomorphic to~$Q_{n-t}$.
 Correlation immune Boolean functions of order~$t$
 are known to be equivalent to simple orthogonal arrays OA$(M,n,2,t)$,
 where $M$ is the number of ones of the function.
Fon-Der-Flaass \cite{FDF:CorrImmBound} proved that the correlation-immunity order of an unbalanced Boolean function in $n$ variables cannot exceed 
$2n/3 - 1$; moreover, any unbalanced Boolean function~$f$ of correlation-immunity order $2n/3 - 1$ corresponds
to an equitable $2$-partition of the $n$-cube $Q_n$ with quotient matrix $[[a,b],[c,d]]$, 
where $a+b=c+d=n$, $a-c=-n/3$, and the number of ones of~$f$ is $2^n c/(b+c)$.
Nowadays, there are three known families of quotient matrices corresponding to such functions: 
$[[0,3T],[T,2T]]$, $[[T,5T],[3T,3T]]$
(found in \cite{Tarannikov2000}), $[[3T,9T],[7T,5T]]$ (found in \cite{FDF:12cube.en}). 
For each of the matrices 
$[[0,3],[1,2]]$, $[[1,5],[3,3]]$, and $[[0,6],[2,4]]$, a function
is unique up to equivalence.
Kirienko~\cite{kirienko2002} found that there are exactly~$2$ inequivalent unbalanced Boolean functions in $9$ variables  attaining the bound on the order of correlation immunity (the corresponding quotient matrix is $[[0,9],[3,6]]$).   
Fon-Der-Flaass started the investigation 
of equitable $2$-partitions of~$Q_{12}$ attaining 
the correlation-immunity bound:
in~\cite{FDF:12cube.en}, it was shown that equitable partitions 
with quotient matrix $[[1,11],[5,7]]$
do not exist, while  equitable partitions with quotient matrix $[[3,9],[7,5]]$ were built. In~\cite{KroVor:2020}, combining the theoretical and computational
approaches, equitable partitions with quotient matrices $[[0,12],[4,8]]$ and $[[3,9],[7,5]]$ were classified up to equivalence. The current work is aimed on 
the remaining case $[[2,10],[6,6]]$.
This case additionally motivated 
by the inner order~$2$ of the first cell 
of such a partition, which means that this cell 
induces a $2$-factor in~$Q_n$, i.e., 
the union of disjoint cycles at distance at least~$2$ from each other.
As we will see from the result of our classification,
the length of cycles can vary over the all family of such partitions, 
as well as in one partition.

Throughout this paper, we use the notation 
$S=[[S_{++},S_{+-}],[S_{-+},S_{--}]]=[[2,10],[6,6]]$
and the term $S$-partition in the meaning
``equitable partition of the $12$-cube with quotient matrix~$S$''.

The results of the present paper can be treated in terms 
of orthogonal arrays. 
A binary \emph{orthogonal array} OA$(M,n,2,t)$
is a multiset of size 
$M$ of vertices on the $n$-cube
such that every $(n-t)$-subcube contains exactly $M/2^t$ of them.
Simple orthogonal arrays (i.e., without multiplicity larger than~$1$)
naturally correspond to $t$-th order correlation immune
Boolean functions in $n$ arguments with $M$ ones: 
the characteristic function of a binary simple orthogonal array of strength~$t$
is a $t$th order correlation immune Boolean function, and vice versa.


In Section~\ref{s:class}, we describe the computer-aided 
classification of equitable \linebreak[5] $[[2,10],[6,6]]$-partitions of 
the $12$-cube. The general approach for the classification
of equitable $2$-partitions of the hypercube is described in
Sections~\ref{s:local} and~\ref{s:covering}; 
the approach is the further development of the algorithms
used in~\cite{KroVor:2020} and~\cite{Kro:OA13}.
However, straightforward applying this approach
to the current parameters requires unreasonable computational resources.
In Section~\ref{s:reduce}, we discuss the possibility
to reduce the amount of computation
by dividing the considered class of equitable partitions
into subclasses, ``square'' and ``non-square'' partitions, 
``heavy'' and ``light'' partitions.
The classification of equitable partitions 
from different subclasses is described in 
Sections~\ref{s:sq}, \ref{s:nl}, and~\ref{s:ns}.

In Section~\ref{s:prop}, we discuss the results of the classification and properties of the found equitable partitions and orthogonal arrays,
including the lengths of the induced cycles
(Section~\ref{s:cycle}),
the order of the automorphism group (Section~\ref{s:aut}),
the Fourier coefficients and being so-called 
``almost orthogonal array'' of strength~$8$ 
(Section~\ref{s:F}; spoiler: only $2$ of $103$ 
 simple OA$(1536,12,2,7)$ and only $1$
 of $16$ OA$(1024,12,2,7)$ satisfy this nice property),
 relations with known constructions (Section~\ref{s:known}).
 
 In Section~\ref{s:related}, we discuss related objects.
 In Section~\ref{s:split},
 equitable $3$-partitions
 with quotient matrix $[[2,6,4],[6,2,4],[6,6,0]$
 are considered 
 (equivalently, pairs of disjoint  simple OA$(1536,12,2,7)$;
 we find that there are $40$ equivalence classes of such pairs),
 which can also be treated as $3$-valued functions
 attaining the correlation-immunity bound.
 In Section~\ref{s:n-s}, we discuss the existence of
 non-simple OA$(1536,12,2,7)$.
 Non-simple orthogonal arrays cannot be treated as Boolean functions;
 however, the correlation-immunity bound 
 can be generalized to non-simple binary orthogonal arrays~\cite{Khalyavin:2010.en}, and the strength $7$ of OA$(1536,12,2,7)$ attains that 
 generalised bound, which motivates their further study.

\section{Classification}\label{s:class}

The general classification approach we use 
is similar to the one used in~\cite{Kro:OA13}
for $[[0,13],[3,10]]$-partitions of $Q_{13}$,
which, in its turn, is a development
of the algorithm used in~\cite{KroVor:2020} for $[[0,12],[4,8]]$-partitions of $Q_{12}$.
The description of the algorithm 
(in particular, the definition of local partitions) 
in the following two subsections
slightly differs from that in~\cite{Kro:OA13} because 
we now have no zeros in the quotient.
However, the main difference and originality of the current search 
is described in Subsection~\ref{s:reduce},
where the class of all $[[2,10],[6,6]]$-partitions
is divided into three subclasses, 
according to containing special configurations in the partition.
As shown preliminary experiments, without this subdivision,
the classification would hardly be possible with modern
computational possibilities.

Our goal is to classify the $[[2,10],[6,6]]$-partitions of $Q_{12}$
up to equivalence, so the definition of the equivalence plays a key role.
Two subsets $C$ and $C'$ of $\{0,1\}^n$ are \emph{equivalent}
if there is an automorphism $\pi$ of the graph $Q_n$
such that $\pi(C)=C'$.
Two $2$-partitions 
$(C_+,C_-)$ and $(C'_+,C'_-)$ of $\{0,1\}^n$ 
are \emph{equivalent} if $C_+$ is equivalent to $C'_+$. 
(In this paper, we compare only equitable $2$-partitions with the same asymmetric
quotient matrix, so the first cell of one partition cannot be equivalent to the second cell of the other).
An \emph{automorphism} of a
subset $C$ of $\{0,1\}^n$ (and of the partition $(C,\{0,1\}^n\backslash C )$ )
is an automorphism $\pi$ of $Q_n$ such that 
$\pi(C)=C$. 
The set $\mathrm{Aut}(C)$ of all automorphisms of $C$ 
forms a group with respect to the composition,
the \emph{automorphism group} of $C$.
The number of $2$-partitions
equivalent to a given $2$-partition $(C_+,C_-)$ 
is calculated as $|\mathrm{Aut}(\{0,1\}^n)| / |\mathrm{Aut}(C_+)| $,  
where $\mathrm{Aut}(\{0,1\}^n)$ is the set of automorphisms of $Q_n$,
$ |\mathrm{Aut}(\{0,1\}^n)| = 2^n\cdot n! $.

\subsection{Local partitions}\label{s:local}
For classification by exhaustive search, 
we use an approach based 
on the local properties of the equitable partitions.
We define objects that satisfy that properties on the words
of small weight
(the \emph{weight} of a binary word is the number of ones in it).
Say that the pair of disjoint sets $P_+$,
$P_-$ 
of vertices of $Q_{12}$ is an \emph{$(r_0,r_1)$-local partition}
(sometimes, we will omit the parameters $(r_0,r_1)$)
if 
\begin{itemize}

 \item[(I)] 
 $P_+ \cup P_-$ 
are the all words starting with $0$
and having weight at most $r_0$
or starting with $1$
and having weight at most $r_1$;

\item[(II)] 
$P_+$ contains the all-zero word $\bar 0$;

\item[(III)] 
$P_-$ contains $100000000000$;

\item[(IV)] 
the neighborhood of every vertex 
$\bar v=(v_1,\ldots,v_{12})$
of weight less than $r_{v_1}$ satisfies the local condition 
from the definition 
of an equitable partition 
with quotient matrix 
$[[2,10],[6,6]]$ 
(that is, if $\bar v\in P_+$ then 
$\bar v$ has $2$ neighbors in $P_+$
and $6$ neighbors in $P_+$.
if $\bar v\in P_-$ then
the neighborhood has exactly $6$
elements in $P_+$ and $6$ in $P_-$);

\item[(V)] 
Every vertex from $P_+$ has at most $2$ neighbors in $P_+$.
\end{itemize}

Two $(r_0,r_1)$-local partitions
$(P_+,P_-)$ and $(P'_+,P'_-)$
are \emph{equivalent} (\emph{$(r_0,r_1)$-equivalent})
if there is a permutation
of coordinates that fixes the first
coordinate and sends
$P_+$ to $P'_+$.

By an $r$-local partition, we mean an $(r,r)$-local partition, 
but the equivalence for this notion is counted differently:
Two $r$-local partitions
$(P_+,P_-)$ and $(P'_+,P'_-)$
are \emph{equivalent} (\emph{$r$-equivalent})
if there is a permutation
of coordinates that sends
$P_+$ to $P'_+$.
So, equivalent $(r,r)$-local partitions are necessarily $r$-equivalent 
(equivalent as $r$-local partitions),
but not vice versa.

The general approach is to classify 
all inequivalent 
$(r_0,r_1)$-local partitions subsequently
for $(r_0,r_1)$ equal
$(2,2)$, $(2,3)$, $(3,3)$, $(3,4)$, $(4,4)$, $(12,12)$, where 
$(12,12)$ corresponds to the complete equitable partitions.
In an obvious way, every equitable partition $(C,\overline C)$ such that $\bar 0\in C$ includes
a $(4,4)$-local partition $(P_+^{(4,4)},P_-^{(4,4)})$, 
$P_+^{(4,4)} \subset C$ and $P_-^{(4,4)} \subset \overline C$, 
every $(4,4)$-local partition $(P_+^{(4,4)},P_-^{(4,4)})$ includes
a $(3,4)$-local partition $(P_+^{(3,4)},P_-^{(3,4)})$, 
$P_+^{(3,4)} \subseteq P_+^{(4,4)}$ and $P_-^{(3,4)} \subseteq P_-^{(4,4)}$, 
and so on.
So, the strategy is to reconstruct, 
in all possible ways, 
a $(r_0,r_1)$-local partition from
each of the inequivalent $(r_0-1,r_1)$-local 
or $(r_0,r_1-1)$-local partitions,
and then to choose and keep only inequivalent solutions,
one representative for each equivalence class found. 
The details of the reconstruction are described in the next subsection
(except the reconstruction from $(4,4)$ to $(12,12)$, 
which is straightforward from the correlation immunity), 
followed by Section~\ref{s:reduce}, 
where we discuss how to reduce the amount of computations
for the concrete quotient matrix $[[2,12],[6,6]]$.

\subsection{A reconstruction step}\label{s:covering}
Assume that we have a $(r_0-1,r_1)$-local partition
$(P_+,P_-)$ (the case $(r_0,r_1-1)$ is considered similarly).
To find all possible $(r_0,r_1)$-local partitions $(R_+,R_-)$
that include it, 
we construct the following instance of the exact cover problem.
\begin{itemize}
 \item Among all words $\bar x$ of weight $r_0$ and starting from $0$, 
we choose only those whose inclusion in the first cell
does not contradict (V), for $\bar x$ itself and for all its neighbors.
Call them \emph{candidates} (for the inclusion in $R_+\backslash P_+$).
 \item Next, for each words $\bar v$ of weight $r_0-1$ and starting from $0$,
we count the number $\alpha_{\bar v}$
of its neighbors to add to the first cell for (IV) to be satisfied.
  If $\bar v \in P_i$, $i\in\{+,-\}$, 
then $\alpha_{\bar v}=S_{i+}-\beta_{\bar v}$, where $\beta_{\bar v}$
is the number of neighbors of  $\bar v$ that are already in $P_+$.
\item We construct a $0,1$-matrix $M=(m_{\bar x,\bar v})$ 
whose rows are indexed by candidates, 
columns are indexed by the words of weight $r_0-1$ starting from $0$,
and $m_{\bar x,\bar v}=1$ if and only if $\bar x$, $\bar v$ are neighbors.
\item Now, finding $(r_0,r_1)$-local continuation of $(P_+,P_-)$ is equivalent
to finding a set of rows of $M$ whose sum is the row $(\alpha_{\bar v})$.
This is an instance of the so-called exact multiple cover problem 
(multiple, because the coefficients $\alpha_{\bar v}$ can be larger than $1$),
and can be solved with an appropriate software; 
we use \texttt{libexact}~\cite{KasPot08}.
\item For each exact cover found, we form $R_+$ by adding 
the indices of the chosen rows to $P_+$,
while $R_-$ is found as the complement of $R_+$.
\end{itemize}

An important step, traditionally called the \emph{isomorph rejection},
is choosing nonequivalent representatives from the set of
all found solutions, intermediate or final. 
A standard technique to deal with equivalence, 
see~\cite[Sect.\,3.3]{KO:alg},
is to represent sets (in our case, $R_+$) 
by graphs and use a software~\cite{nauty2014}
recognizing the isomorphism of graphs. 
At this point, it is possible to make some partial \emph{validation} of 
the intermediate results by double-counting the total number of found
continuations $(R_+,R_-)$ of $(P_+,P_-)$ with the help 
of the orbit-stabilizer theorem, see~\cite[Sect. 10.2]{KO:alg} for the general strategy. Here, we do not stop on details of the isomorph rejection and the validation, because they are the same as in~\cite{Kro:OA13},
and focus on the improvements specific for 
the current parameters of equitable partitions.

\subsection{How to reduce the amount of computation}\label{s:reduce}
In the case when $r_0=r_1$, it is sufficient to continue calculation for a representative
of every equivalence class of $r$-local partitions, $r:=r_0=r_1$ 
(we apply this strategy for $r=2$).
However, the choice of the representative becomes important if
we want to make the experiment repeatable with the same intermediate numerical results:
for different representatives of the same equivalence class $E$
of $r$-local partitions, the number of $(r,r+1)$-local continuations can be different,
but it is the same for equivalent $(r,r)$-local partitions.
So, for the subdivision $\{E_1,\ldots,E_s\}$ of $E$ into $(r,r)$-equivalence classes,
we need to choose one $(r,r)$-equivalence class explicitly.
For this reason, 
we require that the representative chosen for further calculation
must have the minimum number of inequivalent $(r,r+1)$-local continuations.
If, with this condition, 
there is still more than one candidate $(r,r)$-equivalence class,
then we choose the one with the minimum number 
of inequivalent $(r+1,r+1)$-local continuations, then $(r+1,r+2)$,
and so on.
This approach implies that we produce some redundant computation 
(to make the choice, we need to find the  $(r,r+1)$-local
continuations for each subclass $E_1$, \ldots, $E_s$, 
and sometimes need to find the $(r+1,r+1)$-local continuations
for more than one of them,
while only one of them is used for the classification), 
but makes the choice explicit and the experiment
completely repeatable.
It reduces the total number of calculations 
when $s$ is large (up to $10$); we essentially use this 
approach in Section~\ref{s:sq}.

The general algorithm described above 
works faster 
than the similar the algorithm based 
on $r$-local partitions only, 
but still very heavy for the considered parameters,
without taking into account the specific properties
of the classified partitions.
To make the classification doable,
we distinguish the following subfamilies of the 
family of all $[[2,10],[6,6]]$-equitable partitions
of the $12$-cube.

We say that a $[[2,10],[6,6]]$-equitable (local) 
partition $(C_+,C_-)$ of $Q_{12}$
is \emph{square} if there are four vertices of $C_+$
that induce a square subgraph of $Q_{12}$ 
(for square local partitions, we additionally require one of these vertices to be the all-zero word).
A $[[2,10],[6,6]]$-equitable (local) 
partition $(C_+,C_-)$
is \emph{square-free} if there are no four vertices of $C_+$
inducing a square subgraph of $Q_{12}$.

We say that a $[[2,10],[6,6]]$-equitable
partition $(C_+,C_-)$ of $Q_{12}$
is \emph{heavy} (\emph{light})
if there is (there is no) a subgraph of $Q_{12}$
isomorphic to the cube $Q_3$ and having at least 
$5$ five vertices from $C_+$.

The classification is divided into the following three
stages; each stage is considered in a separate subsection below.
At stage one, we classify the square partitions. 
Stage two deals with heavy partitions,
and its main result is that all heavy 
$[[2,10],[6,6]]$-equitable partitions of $Q_{12}$
are square.
Stage three is the classification of 
the square-free partitions; 
by the results of the second stage they are all light.

\subsection{The square partitions}\label{s:sq}
Up to equivalence, any square $S$-partition or square local partition contains the four words $\bar0=0...000$, 
$\bar e_{12}=0...001$, $\bar e_{11}=0...010$, $\bar e_{11}+\bar e_{12}=0...011$ in its first cell.
\begin{lemma} There are $66462606$ $2$-local partitions
$(C_+,C_-)$ with 
$\bar0$, 
$\bar e_{12}$, 
$\bar e_{11}$, 
$\bar e_{11}+\bar e_{12}\in C_+$,
 which are partitioned into $60$
 $2$-equivalence classes or 
 $286$ $(2,2)$-equivalence classes.
\end{lemma}
\begin{proof}
Assume that we have a $2$-local partition $(C_+,C_-)$ with 
$\bar0$, 
$\bar e_{12}$, 
$\bar e_{11}$, 
$\bar e_{11}+\bar e_{12}\in C_+$.
Every weight-$2$ word from $C+$ is adjacent to exactly two 
words from $\bar e_{1}$, \ldots, $\bar e_{10}$.
On the other hand, 
every word from $\bar e_{1}$, \ldots, $\bar e_{10}$
has exactly $5$ weight-$2$ neighbors in $C_+$ (totally
$S_{-,+} = 6$ neighbors in $C_+$, but one of them is $\bar 0$).
Clearly, this incidence structure corresponds to a $5$-regular
graph on the $10$ vertices $\bar e_{1}$, \ldots, $\bar e_{10}$.
It is easy to see that this correspondence with the class
of $5$-regular
graphs of order $10$ is one-to-one; 
moreover, two graphs are isomorphic if and only if the corresponding $2$-local partitions are equivalent.
The number of $5$-regular
graphs on $10$ labelled vertices is $66462606$, see
 http://oeis.org/A059441,
 and the number of their isomorphism classes is $60$, 
 see http://oeis.org/A006821.
 The number of $(2,2)$-equivalence classes is counted computationally;
 it equals the number of vertex orbits summed over the $60$
 non-isomorphic $5$-regular
graphs of order $10$ 
(equivalently, the number of non-isomorphic $5$-regular graphs
of order $10$ 
with one marked vertex, 
which corresponds to the first coordinate of the $(2,2)$-local partition).
\end{proof}

Next, realizing the strategy described in Section~\ref{s:reduce},
we choose $60$ representatives of 
$(2,2)$-local partitions, and for each of them find
all non-equivalent $(2,3)$-local, $(3,3)$-local, 
$(3,4)$-local, $(4,4)$-local, and finally $S$- partitions.
The local partitions that continue the chosen $60$ representatives
will be called \emph{leading}, in contrast to the 
local partitions that continue the rest $286-60$ $(2,2)$-equivalence classes. Up to equivalence, we have the following 
numbers of square partitions:

\begin{itemize}
 \item there are $37141023$ (few CPU days) 
 square $(2,3)$-local partitions,
 $4979729$ of them are leading;
 
 \item there are $659276500$ square $(3,3)$-local partitions,
 $94275707$ (736 CPU days) of them are leading; there are $65945212$ $3$-local partitions (comparing $94275707$ and $65945212$, we see that applying the strategy described in Section~\ref{s:reduce} 
 can reduce the number of further calculation by factor about $1.4$; it was decided not to implement this improvement);
 
 \item there are $16535880038$ (1795 CPU days) 
 leading square $(3,4)$-local partitions;
 
  \item there are $3111$ square $(4,4)$-local partitions,
 $667$ (2322 CPU days) of them are leading; 
 there are $429$ square $4$-local partitions, and all of them continue to $S$-partitions;
 
 \item there are $77$ square $S$-partitions.
\end{itemize}


\subsection{The heavy partitions}\label{s:nl}

Define the $(1,2)$-local partition $(B_+,B-)$ where
$B_+ = \{ \bar e_1+\bar e_2$, 
$\bar e_2$, 
$\bar 0$, 
$\bar e_3$,
$\bar e_3+\bar e_1$,
$\bar e_1+\bar e_4$, $\bar e_1+\bar e_5$, $\bar e_1+\bar e_6\}$.

\begin{lemma}
Every heavy $S$-partition is equivalent to a partition
continuing the $(1,2)$-local partition $(B_+,B-)$.
\end{lemma}
\begin{proof}
(i) It is easy to observe that if $C$
is a vertex set of $Q_3$ of cardinality
at least $5$ and such that the induced subgraph
has no vertices of degree more than two,
then $C$ induces either a $5$-path or a $6$-cycle,
being equivalent to either
$$\{110,010,000,001,101\}\quad \mbox{or}\ \{110,010,000,001,101,111\},$$
respectively.

(ii) Let $(C_+,C_-)$ be a heavy $S$-partition.
Without loss of generality, 
According to the definition of heavy $S$-partition,
we can assume without loss of generality that the $8$ vertices
ending by $9$ zeros include at least $5$ vertices from $C_+$.
From (i) we see that these vertices are 
$110\,000000000$, $010\,000000000$, $000\,000000000$, $001\,000000000$, $101\,000000000$,
and may be  $111\,000000000$. It remains to note that the first cell of the corresponding 
$(1,2)$-local partition contains the first $5$ of these vertices
and $3$ more neighbors of $100000000000$, which are $100100000000$, $100010000000$, and $100001000000$,
up to coordinate permutation.
\end{proof}

In contrast to the other subcases, in this section we choose 
the $(1,3)$-local partitions for the next step of the classification
instead of $(2,2)$-local partitions. 
Up to equivalence, we have the following number of partitions
containing 
$\bar e_1+\bar e_2$, 
$\bar e_2$, 
$\bar 0$, 
$\bar e_3$, 
$\bar e_3+\bar e_1$:
\begin{itemize}
 \item there are $178$ $(2,2)$-local partitions ($2$-local partitions);
 \item there are $953730$ (1 CPU minute) $(2,3)$-local partitions;
 \item there are $815364$ (30 CPU days) $(3,3)$-local partitions ($3$-local partitions);
 up to this step, all local partitions are square-free;
 \item there are $2325257827$ $(3,4)$-local partitions, $ 560234770 $ (14 CPU days) of them are square-free;
 \item there are $54$ $(4,4)$-local partitions ($4$-local partitions),
 all of them continue to $S$-partitions, $0$ (7 CPU  days) of them are square-free;
 \item there are $17$ heavy $S$-partitions, none of them is square-free (17 CPU days).
\end{itemize}
\begin{remark}
 In the considered case, the $r$-equivalence coincides with  the $(r,r)$-equivalence
 because the first coordinate is special and cannot be permuted with any other.
\end{remark}

Finally, we have established the following fact, 
which essentially simplifies the next-section classification 
of the square-free $S$-partitions.

\begin{lemma}\label{l:sf-lt}
All square-free $S$-partitions are light. 
\end{lemma}
\begin{remark}
To validate Lemma~\ref{l:sf-lt}, it was sufficient to consider only square-free partitions.
The numbers of partitions without the square-free condition provided above
are based on early computations, which occur to be redundant. The CPU time was not kept;
it is essentially smaller than that for the step in the previous section,
and it is surely possible to save some time if consider only light square partitions in 
Section~\ref{s:sq}. In any case, the largest part of computation is described in the next section.
\end{remark}


\subsection{The square-free partitions}\label{s:ns}

\begin{lemma}
Every square-free $S$-partition is equivalent to a partition
containing 
$\bar e_1+\bar e_2$, 
$\bar e_2$, 
$\bar 0$, 
$\bar e_3$, and
$\bar e_3+\bar e_4$, 
in its first cell $C_+$.
\end{lemma}
\begin{proof}
 Without loss of generality, we can assume that $\bar 0$
 belongs to $C_+$, and its two neighbors in $C_+$ are 
 $\bar e_2$ and $\bar e_3$. The word 
 $\bar e_2 + \bar e_3$ is not in $C_0$ because 
 the partition is square-free.
 Hence, the second $C_+$-neighbor of $\bar e_2$
 is $\bar e_i + \bar e_2$, where $i\ne 2,3$. 
 We can assume $i=1$. 
 Similarly, 
 the second $C_+$-neighbor of $\bar e_2$
 is $\bar e_3 + \bar e_j$, where $j$ is not $2$ or $3$.
 Additionally, $j$ is not $1$, 
 because the partition
 is not light otherwise, 
 contradicting Lemma~\ref{l:sf-lt}.
 So, we can assume $j=4$.
\end{proof}

Based on the lemma above, we will search for
only local partitions having 
$\bar e_1+\bar e_2$, 
$\bar e_2$, 
$\bar 0$, 
$\bar e_3$, and
$\bar e_3+\bar e_4$, 
in the first cell.
Among such partitions,
\begin{itemize}
 \item there are $1786$ equivalence classes
 of
 $(2,2)$-local partitions, $1010$ equivalence classes
 of $2$-local partitions.
\end{itemize}

Realizing the strategy described in Section~\ref{s:reduce},
we can choose $1010$ representatives of 
$(2,2)$-local partitions, and for each of them find
all non-equivalent $(2,3)$-local, $(3,3)$-local, 
$(3,4)$-local, $(4,4)$-local, and finally $S$- partitions.
The local partitions that continue the chosen $1010$ representatives are called \emph{leading}. 
Up to equivalence, we have the following 
numbers of square-free partitions:

\begin{itemize}
 \item there are $226841305$ 
 (few CPU days) 
 square-free $(2,3)$-local partitions,
 $ 77868835 $ of them are light,
 $ 38039061 $ of them are leading;
 
 \item there are $ 166399852 $ 
 light square-free 
 $(3,3)$-local partitions,
 $ 96211116 $ (1500 CPU days) of them are leading; 
 there are $ 83210833 $ $3$-local partitions
 (again, comparing $ 96211116 $ 
 and $ 83210833 $, 
 we find unreasonable applying 
 the strategy of Section~\ref{s:reduce} for $r=3$);
 
 \item there are $23 752 571 733$ 
 light square-free $(3,4)$-local partitions
 $14937782031$ of them (2328 CPU days) 
 are leading;
 
  \item there are $923$  light square-free 
  $(4,4)$-local partitions,
 $663$ (1306 CPU days) of them are leading; 
 there are $490$  light square-free 
 $4$-local partitions,
 $312$ of them are completable to square-free 
 $S$-partitions (the other are also completable
 to $S$-partitions, but with squares);
 
 \item there are $26$ square-free $S$-partitions.
\end{itemize}

\section{Results and properties}\label{s:prop}

As we see in Sections~\ref{s:sq}--\ref{s:ns}, 
there are 
$103$ nonequivalent 
$[[2,10],[6,6]]$-partitions,
$77$ of them have squares in the first cell,
while the remaining $26$ are square-free.
In this section, we discuss properties of the found equitable partitions.

\subsection{Cycle length}\label{s:cycle}
The first cell of a $[[2,10],[6,6]]$-partition induce a disjoint union of cycles in $Q_n$.
These cycles can be of different lengths, the possible values are 
$4$, $8$, $10$, $12$, $16$, $18$, $20$, $24$, $28$, $30$, $32$, 
$36$, $40$, $44$, $48$, $52$, $60$, $88$, $120$.
The following partitions have only one cycle length:

\#1: {\boldmath$4^{384}$} (i.e., $384$ cycles of length $4$);

\#101, \#102, \#103: {\boldmath$8^{192}$};\qquad
\#91, \#93, \#94, \#95: {\boldmath$24^{64}$};

\#78, \#79, \#80, \#81, \#82, \#83, \#84, \#85, \#86, \#87, \#90, \#92: {\boldmath$48^{32}$}.\\
The other partitions without a square in the first cell:

\#88, \#89, \#96, \#97, \#99, \#100: {\boldmath$8^{64} 16^{64}$};\qquad
\#98: {\boldmath$8^{64} 32^{32}$}.\\
Only one partition has more than five different cycle lengths: 

\#60: {\boldmath$4^{32} 18^{8} 20^{32} 30^{8} 36^{4} 60^{4}$}.\\
Two partitions have cycles of the maximum length:

\#69, \#71: {\boldmath$4^{64} 40^{8} 120^{8}$}.\\
Note that the cycle formula does not reflect other parameters
such as the number of cycles that are not translations of each other, 
the number of different directions in a cycle, the diameter of a cycle;
these parameters can vary over partitions with the same cycle formula.

\subsection{Automorphisms}\label{s:aut}
The order of the automorphism group of a partition from the considered class
can possess the values
$8$ (\#64, \#65),  $16$, $32$, $64$, $128$, $160$ (\#69, \#71), $256$, $384$ (\#19), 
$512$, $640$ (\#73, \#74), $768$ (\#95), $1024$, $2048$, $2560$ (\#33), $3072$, $4096$, 
$8192$ (\#15, \#102), 
$12288$ (\#103),
$16384$ (\#8, \#12),
$24576$ (\#101),
$32768$ (\#3),
$983040$ (\#1).
The number of translational automorphisms (i.e., the number of \emph{periods}
$\bar x$: $C_+ + \bar x = C_+$) can be 
$4$, $8$, $16$, $32$ (\#2, \#12, \#13, \#15, \#28), 
$64$ (\#3, \#8, \#101), $128$ (\#1).
The only partitions that have odd-weight (with odd, $5$ or $7$, number of ones) periods 
are \#41, \#43, \#74.
The cell $C_+$ is partitioned into orbits
under the action of $\mathrm{Aut}(C_+)$; 
the following partitions have small number of the orbits:
\#1, \#101 ($1$ orbit), 
\#16, \#95, \#103 ($2$ orbits),
\#3, \#8, \#12, \#32, \#88, \#102 ($3$ orbits),
\#15, \#20, \#28, \#30, \#33, \#96, \#97 ($4$ orbits).
A notable partition is \#95: 
the first cell is divided into only two orbits,
while the induced cycles are of length $24$. 

\subsection{Good orthogonal arrays and Fourier coefficients}\label{s:F}

A multiset $C$ of tuples from $\{0,1\}^n$ is called a \emph{(binary) orthogonal array}
of strength~$t$, OA$(|C|,n,2,t)$, if every $(n-t)$-subcube of $Q_n$ has exactly
$|C|/2^t$ elements of~$C$, taking into account the multiplicities.
If there are no elements of multiplicity more than one (i.e., $C$ can be treated 
as an ordinary set), the orthogonal array is called \emph{simple}.
In other words, a simple OA$(M,n,2,t)$ is a subset of $\{0,1\}^n$ of size~$M$
whose characteristic function is correlation immune of order~$t$.
In particular, $[[2,10],[6,6]]$-partitions correspond to simple OA$(1536,12,2,7)$.
By the bound in~\cite{Khalyavin:2010.en} (see~\cite{FDF:CorrImmBound} for the simple case),
an array of cardinality less than $2^{n-1}$ has strength at most $2n/3-1$ (for $n=12$, at most~$7$). So, among orthogonal arrays of strength~$7$ and size less than $2048$,
there are no orthogonal arrays of strength~$8$. 
It is natural to ask which of them better approximate 
orthogonal arrays of strength~$8$.
We will say that $C$ is an \emph{almost orthogonal array} of strength $t+1$,
OA$(|C|,n,2,t+)$, if it is OA$(|C|,n,2,t)$ and every $(n-t-1)$-subcube of $Q_n$ has exactly
$|C|/2^{t+1}-1$, $|C|/2^{t+1}$, or $|C|/2^{t+1}+1$ elements of $C$. From~\cite{KroVor:2020},
we know that there are two simple OA$(1792,12,2,7)$ and both of them are 
OA$(1792,12,2,7+)$. 
From current computations, we find that there are $103$ simple
OA$(1536,12,2,7)$ and only two of them, \#81 and \#82 
(with $32$ cycles of length $48$), are OA$(1536,12,2,7+)$.
Also, from data in~\cite{KroVor:2020},
we find that among the $16$ OA$(1024,12,2,7)$, there is only one OA$(1024,12,2,7+)$, 
number $16$.

The \emph{Fourier decomposition}
of a real-valued function $f:\{0,1\}^n \to \mathbb{R}$
is the representation of~$f$ in the form
$$ f(\bar x) = \sum_{\bar y \in \{0,1\}^n} \hat f(\bar y) (-1)^{\bar y \cdot \bar x}, $$
where $\bar y \cdot \bar x$ is the \emph{inner product} defined
as $(y_1, ... ,y_n) \cdot (x_1, ... ,x_n) = y_1x_1+ ... + y_nx_n$, and
$\hat f(\bar y)$ are the \emph{Fourier coefficients}
defined uniquely for each $f$ (indeed, the functions 
$\chi_{\bar y}(\bar x) = (-1)^{\bar y \cdot \bar x}$,
$\bar y \in \{0,1\}^n$, form an orthogonal basis of the space of all real-valued
functions on $\{0,1\}^n$). The function $\hat f$ is called the 
\emph{Fourier transform} of $f$.
For an equitable $S$-partition $(C_+,C_-)$, 
it is convenient to consider the Fourier transform
$\hat f_{C_+,C_-}$ of 
$$f_{C_+,C_-}(\bar x)=
\begin{cases}
\phantom{-}S_{+-} & \mbox{if $\bar x \in C_+$} \\
-S_{-+} & \mbox{if $\bar x \in C_-$}
\end{cases}.
$$
It is known that 
\begin{itemize}
 \item [(a)]
 $\hat f_{C_+,C_-}(\bar y)=0$ unless 
$\bar y$ has exactly $(b+c)/2$ ones;
 \item [(b)]
 if $\bar y$ has exactly $(b+c)/2$ ones,
 then 
 $$
 \hat f_{C_+,C_-}(\bar y)=
 2^{(S_{+-}+S_{-+})/2-n}
 \sum_{\bar x \perp \bar y} 
         f_{C_+,C_-}(\bar x),
 $$
 where $\bar x \perp \bar y$ means 
 that $\bar x$ and $\bar y$ have no ones in the same position;
 \item [(c)]
 in particular, every Fourier coefficient 
 is a multiple of 
 $2^{(S_{+-}+S_{-+})/2-n}(S_{+-}+S_{-+})$;
 \item[(d)] 
 $\sum_{\bar y \in \{0,1\}^n} 
        (\hat f_{C_+,C_-}(\bar y))^2 = S_{+-}S_{-+}$.
\end{itemize}
It follows from~(c) and~(d) that
the Fourier transform of every 
$[[2,10],[6,6]]$-partition is integer and the sum of the squares of all coefficients is $60$. In fact, the possible values of coefficients
are $0$, $\pm 1$, $\pm 2$. It can be seen from (b)
that the partitions with all Fourier coefficients in 
$\{0,1,-1\}$ ($60$ non-zero coefficients in total) 
correspond to OA$(1536,12,2,7+)$ 
(equivalence classes \#81 and \#82);
i.e.,
for these partitions,
every $4$-subcube has $11$, $12$, or $13$
elements of the first cell.
There are four inequivalent partitions
with the Fourier coefficients in $\{0,2,-2\}$
($15$ non-zero coefficients), \#1, \#3, \#8, and \#101;
for these partitions, every $4$-subcube has $10$, $12$, or $14$
elements of the first cell.
For the partitions from the other $98$ equivalence classes,
each value from $10$ to $14$ is realized in some $4$-subcube.

\subsection{Known constructions}\label{s:known}
The partitions
\#1, \#3, \#15, and \#103 
(cycle formulas $4^{384}$, 
$4^{256}8^{64}$, $4^{128}8^{128}$, 
$8^{192}$, respectively) are
obtained from the unique 
$[[1,5],[3,3]]$-partition $(P_+,P_-)$ 
by the constructions
\begin{equation}\label{eq:*2}
C_i = \{(\bar x, \bar y ) \mid \bar x + \bar y \in P_i\},
\end{equation}
where the addition is coordinatewise over $\mathbb{Z}_2$ 
for \#1, by pairs of coordinates  over $\mathbb{Z}_4$ for \#103
(binary pairs are treated as elements of $\mathbb{Z}_4$, via the map $00\to 0$, $01\to 1$, $11\to 2$, $10\to 3$),
and of mixed $\mathbb{Z}_2\mathbb{Z}_4$ type for \#3 and \#15
(one or two pairs of coordinates are treated as elements of $\mathbb{Z}_4$, while the remaining $4$ or $2$ coordinates are treated as elements of $\mathbb{Z}_2$).

The next construction is described here without some technical details, which can be found in~\cite{FDF:PerfCol} and~\cite{VorFDF}. 
$[[2,10],[6,6]]$-partitions 
can be constructed in the following way.
At first, we take a 
$[[2,1],[3,0]]$-partition
$(A_+,A_-)$, say $A_-=\{000,111\}$; note that $A_+$
can be partitioned into $3$ edges of the $3$-cube.
Defining
$ B_+ = \{(\bar x, \bar y ) \mid \bar x + \bar y \in A_+\}$
(with the coordinatewise binary addition), we get 
a $[[4,2],[6,0]]$-partition $(B_+,B_-)$,
where $B_+$ can be partitioned into twelve quadruples $Q_1$, \ldots, $Q_{12}$ inducing $2$-subcubes 
of the $6$-cube.
The next step is a special case of the construction
in~\cite[Sect.~3]{FDF:PerfCol},
with the additional possibility to switch some subsets
observed in~\cite{VorFDF}.
Applying the doubling construction again
results in a $[[8,4],[12,0]]$-partition $(C_+,C_-)$,
where $C_+$ can be split into $C'_+$ and $C''_+$
to form a $[[2,6,4],[6,2,4],[6,6,0]$-partition $(C'_+,C''_+,C_-)$.
There are many ways to split:
each set 
$\{(\bar x, \bar y ) \mid \bar x + \bar y \in Q_i\}$,
is divided into two subsets $Q'_i$ and $Q''_i$ according 
to the formula in~\cite{FDF:PerfCol}, 
but, as noted in~\cite{VorFDF}, for each $i$
we are free to choose which of $Q'_i$, $Q''_i$ is included 
in $C'_+$ and which in $C''_+$. So, we can get 
$2^{12}$ different $[[2,6,4],[6,2,4],[6,6,0]$-partitions $(C'_+,C''_+,C_-)$.
It is straightforward from the quotient matrix that for each of them, 
$(C'_+,C''_+ \cup C_-)$ is a $[[2,10],[6,6]]$-partition.
In such a way, we get $8$ equivalence classes of $[[2,10],[6,6]]$-partitions,
\#1, \#2, \#3, \#5, \#7, \#8, \#16, and \#101. 
With some accuracy, this switching approach 
can be combined with varying the addition 
as in the previous paragraph, 
but this possibility was not yet developed.
 
For \#1, \#73 (cycle formula $4^{64} 20^{16} 60^{16}$), 
\#74 ($4^{64} 10^{16} 20^8 30^{16} 60^8$), 
the first cell of a partition
can be obtained as the projection (puncturing) of the first cell of the unique 
$[[0,13],[3,10]]$-partition (the construction is given in \cite{FDF:PerfCol}, the uniqueness is established in~\cite{Kro:OA13}) in one of the directions (one direction corresponds to \#1, six directions to \#73, and six to \#74).

\section{Related structures}\label{s:related}
\subsection
[Splitting into equitable 3-partitions]
{Splitting into equitable $3$-partitions}
\label{s:split}
Here, we discuss pairs of disjoint orthogonal arrays
OA$(1536,12,2,7)$. The complement of the union of two
disjoint OA(1536,12,2,7) is necessarily an OA$(1024,12,2,7)$,
and these three arrays necessarily form an equitable partition
with quotient matrix
$[[2,6,4],[6,2,4],[6,6,0]]$.
The easiest way to classify such partitions is, starting
from OA$(1536,12,2,7)$, to find all OA$(1024,$ $12,$ $2,$ $7)$ disjoint
to it. 
This can be done with the same approach as
the classification of OA$(1024,12,2,7)$ in~\cite{KroVor:2020}.
It can be considered as a simplified version of the
computational approach considered in the present paper,
and we do not discuss the details here;
the only thing we note is that the number of solutions 
at each step is relatively small, and isomorph rejection
is not necessary until the final step.
As a result, we find that for only~$36$ 
of~$103$ inequivalent OA$(1536,12,2,7)$
the complement can be split
into  OA$(1536,12,2,7)$ and OA$(1024,12,2,7)$.
For~$6$ of~$103$ inequivalent OA$(1536,12,2,7)$
(No 1, 3, 8, 15, 101, 103),
the complement can be split in~$5$ ways,
(No 2, 4--7, 9--14, 16--21, 27--30, 32, 88--89, 96--100, 102)
for the rest~$67$ OA$(1536,12,2,7)$, the complement
in unsplittable.
Totally, there are $40$ inequivalent pairs 
of disjoint OA$(1536,12,2,7)$ (essentially, 
equitable $[[2,6,4],[6,2,4],[6,6,0]]$-partitions).
In $38$ of them, the two OA(1536,12,2,7) are equivalent to each other; for $2$ pairs, they are not equivalent.
Only $5$ (No 1, 2, 3, 15, and 16, according to~\cite{KroVor:2020}) of the $16$ inequivalent OA(1024,12,2,7) can occur as the complement of two disjoint OA(1536,12,2,7).
See more details in the appendix.

\subsection[Non-simple OA(1536,12,2,7)]
{Non-simple \boldmath OA$(1536,12,2,7)$}\label{s:n-s}
In this section, we discuss the existence on non-simple 
orthogonal arrays with parameters OA$(1536,12,2,7)$.
One such array can be constructed by~\eqref{eq:*2}
from the following non-simple OA$(24,6,2,3)$:
the $20$ weight-$3$ words of length~$6$
are taken with multiplicity~$1$
and two words $000000$ and $111111$, with multiplicity~$2$.
It is not difficult to observe that different meanings
of~``$+$'' in~\eqref{eq:*2} result in equivalent arrays.
A nice property of this OA$(24,6,2,3)$ and the corresponding OA$(1536,12,2,7)$
is that they are also related to equitable partitions:
the non-simple OA$(24,6,2,3)$ and the corresponding OA$(1536,12,2,7)$ 
are obtained from equitable partitions
$(C_0,C_1,C_2,C_3)$ of~$Q_{6}$ and~$Q_{12}$ with quotient matrices
\begin{equation}\label{eq:SS}
 \begin{pmatrix}
0 & 6 & 0 & 0 \\
1 & 0 & 5 & 0 \\
0 & 2 & 0 & 4 \\
0 & 0 & 6 & 0
\end{pmatrix}
\qquad \mbox{and}\qquad
\begin{pmatrix}
0 & 12 & 0 & 0 \\
2 & 0 & 10 & 0 \\
0 & 4 & 0 & 8 \\
0 & 0 & 12 & 0
\end{pmatrix},
\end{equation}
respectively, by taking~$C_0$ with multiplicity~$2$
and~$C_3$ with multiplicity~$1$.
The next theorem shows that we cannot
construct inequivalent non-simple OA$(1536,12,2,7)$
related to such equitable partitions.
\begin{theorem}
 Up to equivalence, there is only one equitable partition of~$Q_{12}$
 with quotient matrix~\eqref{eq:SS}.
\end{theorem}
\begin{proof}
 Without loss of generality, we assume that the all-zero word~$\bar0$
 is in~$C_0$. It follows that all weight-$1$ words are in~$C_1$.
 Since every vertex of~$C_1$ has exactly~$2$ neighbors in~$C_0$,
 we see that there are~$6$ weight-$2$ words in~$C_0$
 and no two of them have~$1$ in a common position.
 W.l.o.g., they are of form $(x,x)$, where~$x$ is a weight-$1$ word of length~$6$.
 They have $10\cdot 6 = 60$ neighbors of weight~$3$ in total, 
 which are necessarily in~$C_1$.
 The $60$ other weight-$2$ words (that are not of form $(x,x)$) 
 are in~$C_2$.
 Each of them has exactly~$8$ neighbors in~$C_3$, which 
 are necessarily of weight-$3$, while every weight-$3$ word in~$C_3$
 has exactly~$3$ weight-$2$ neighbors in~$C_2$.
 We find that there are exactly $60\cdot 8/3 = 160$
 weight-$3$ words in~$C_3$. Since $\binom{12}{3}=60+160$,
 we see that all the weight-$3$ words that are not neighbors
 of weight-$2$ words of~$C_0$ (that is, the weight-$3$ words of form $(x,y)$
 where~$x$ and~$y$ have no $1$ in a common position) are in~$C_3$.
 Now, consider the words of weight~$4$ in~$C_0$. They can only be of form
 $(x,x)$, because every other weight-$4$ word has neighbors in~$C_3$ (of weight~$3$). On the other hand, by counting the number 
 of $C_0$-neighbors for the $C_1$-neighbors of weight~$3$,
 we conclude that all words of form $(x,x)$ and weight~$4$
 (as well as weight~$0$ and~$2$) are in~$C_0$.
 The rest of the proof consists of simple numbered claims.

 Claim~1: \emph{if $x$, $y$, and $z$ belong to $C_0$ and $y$, $z$ 
 are at distance~$2$ from~$x$, then $x+y+z \in C_0$}. Proof:
 from the consideration above, we see that (*) holds for $x=\bar 0$.
 Similarly, it holds for every other~$x$ from~$C_0$.
 
 Denote by $C^{(i)}$ the set of all length-$12$ binary words of form $(x,y)$ such that $x+y$ has weight~$i$.
 
 Claim~2: \emph{$C^{(0)}$ is a subset of $C_0$}. Proof: straightforwardly
 from Claim~1. 
 
 Claim~3: \emph{$C^{(i)}$ has exactly~$2i$ neighbors in~$C^{(i-1)}$
 and $12-2i$ neighbors in~$C^{(i+1)}$, $i=0,\ldots,6$}. Proof: straightforwardly
 from the definition.
 
 Claim~4: \emph{$C^{(1)} \subset C_1$, $C^{(2)} \subset C_2$, $C^{(3)} \subset C_3$, 
 $C^{(4)} \subset C_2$, $C^{(5)} \subset C_1$, $C^{(6)} \subset C_0$}. 
 Proof: straightforwardly  from Claim~2, Claim~3, and the quotient matrix.
 
 Claim~5: $C_0=C^{(0)} \cup C^{(6)}$, $C_1=C^{(1)} \cup C^{(5)}$, $C_2=C^{(2)} \cup C^{(4)}$, $C_3=C^{(3)}$.
 
 Finally, we see that the partition is reconstructed up to equivalence.
\end{proof}

\begin{problem}
 Are there inequivalent non-simple OA$(1536,12,2,7)$? 
 Is there a non-simple OA$(1792,12,2,7)$? 
 Are there orthogonal arrays OA$(M,12,2,7)$ with
 $M<2048$, $M\not\in\{1024,1536,1792\}$ (such arrays cannot be simple)?
\end{problem}

\appendix
\section{List of partitions}
Below, we list representatives of all $103$ equivalence classes 
of equitable partitions with quotient matrix $[[2,10],[6,6]]$,
which correspond to simple OA$(1536,12,2,7)$,
and $40$ classes of equitable partitions with quotient matrix 
$[[2,6,4],[6,2,4],[6,6,0]]$ 
(partitions into two OA$(1536,12,2,7)$
and one OA$(1024,12,2,7)$).

We first describe the list of equitable $[[2,10],[6,6]]$-partitions.
It would take unreasonably much space
to list each partition completely. We use some easy-to-recover form,
representing each partition $(C_+,C_-)$ by the characteristic function $\chi_{C_+}$ of the first cell,
restricted by the weight-$4$ words only.
The values of this function on the words of weight $3$ 
(similarly, weight $2$, weight $1$, and $0$) can be easily reconstructed by the following rule:
if a weight-$3$ word has at most $2$ weight-$4$ neighbors in $C_+$, then it belongs to $C_+$,
otherwise it is in $C_-$. The values on each word $x$ of weight $w=5$ (then, $6$, $7$, \ldots, $12$) 
can be found using the correlation immunity of $\chi_{C_+}$: 
the number of ones in the set $\{y \mid y\preceq x\}$ is $\frac {6}{10+6}\cdot 2^w$.
The list of values on the $495$ weight-$4$ words, listed in the lexicographic 
ordering $000000001111$, $000000010111$, \ldots, $111100000000$, 
is represented in the hexadecimal form, each symbol corresponding to $4$ binary
values, except the first symbol of the sequence, which corresponds to $3$ binary values 
(note that the first value can be $0$, representatives No 75, 76, 77);
each list is given in two lines following the class number with period (1. \ldots, 2. \ldots, 103. \ldots).

The representatives are lexicographically ordered 
in the following manner:
for each partition, we order the first cell lexicographically,
forming a tuple from $1536$ words;
then, such lists are compared in the lexicographical way.
Each equivalence class is represented by the lexicographically
first partition, and all $103$ representatives of different classes
are ordered as well. In particular, the first $77$ contain
$0...0$, $0...01$,  $0...010$,  $0...011$, 
forming a square, and the last $26$, square-free,
contain
$0...0$, $0...01$,  $0...010$,  $0...0101$, $0...01010$.
Note that the lexicographic order between functions on 
the whole $12$-cube is not kept after restricting 
by the weight-$4$ words only;
that is why the list of sequences below does not look ordered.

If in an equitable $[[2,10],[6,6]]$-partition $(C_+,C_-)$
the cell $C_-$ is splittable into cells $C_*$, $C_\bullet$ of 
an equitable $[[2,6,4],[6,2,4],[6,6,0]]$-partition,
then we also represent the smallest cell $C_\bullet$
over all inequivalent ways to split.
The algorithm to reconstruct all vertices of~$C_\bullet$ from 
the weight-$4$ vertices, which are listed, is similar 
to the one for~$C_+$:
if a weight-$3$ (similarly, for weight-$2$ and weight-$1$) 
word has no weight-$4$ neighbors in $C_\bullet$,
then 
then it belongs to~$C_\bullet$,
otherwise it is in~$C_+$ or~$C_*$.
For words of weight $w=5$ (then, $6$, $7$, \ldots, $12$),
the words of~$C_\bullet$ are determined from the orthogonal array property.
Each list 
representing~$C_\bullet$
is given in two lines following the class number
(of equitable $[[2,6,4],[6,2,4],[6,6,0]]$-partitions)
in parentheses.
The numbers in parentheses can repeat because 
inequivalent equitable $[[2,10],[6,6]]$-partition
(representatives No 8 and No 15, No 101 and No 103) 
can be split into equivalent $[[2,6,4],[6,2,4],[6,6,0]]$-partitions. So, there are $40$ inequivalent equitable $[[2,6,4],[6,2,4],[6,6,0]]$-partitions, while our list contains $42$ representatives.
The number in the brackets in the right of the list indicates 
the equivalence class of the equitable $[[0,12],[4,8]]$-partition
$(C_\bullet,C_+ \cup C_*)$, according to the classification in~\cite{KroVor:2020};
the following number is the number of ways a code equivalent
to~$C_\bullet$ can be embedded in the complement of~$C_+$.

\small
\begin{verbatim}
  1. 7321c79e3c79e384ce1ffe00000000061ffe0000000007099cc3200001f807
     f00810e670c30e670c20600001f807f00810e670c30e670c207099cc321819
 (1) 00003861c3861c000000000f199e330000000f199e330000033cc70c380000
     00066f198000f1980019870c38000000066f198000f1980019800033cc0000 [1]:5
  2. 7321c79e3c79e384ce1ffe00000000061ffe00000000070aad45200001f807
     f00810d568a30e670c20600001f807f00810e670c30d568a2070aad4521819
 (2) 000000000000000001e001e07f80ff79e001e07f80ff780000001e001e07f8
     0ff7800000000000001e1e001e07f80ff7800000000000001e00000001e1e0 [1]:1
  3. 7321c79e3c79e384ce1ffe00000000061ffe00000000070e670c200001f807
     f0081099cc330e670c20600001f807f00810e670c3099cc32070e670c21819
 (3) 00003861c3861c000000000f199e330000000f199e3300f19800070c380000
     0006600033ccf1980019870c38000000066f19800000033cd98f1980000000 [3]:4
 (4) 000000000000000001e001e07f80ff79e001e07f80ff780000001e001e07f8
     0ff7800000000000001e1e001e07f80ff7800000000000001e00000001e1e0 [1]:1
  4. 7321c79e3c79e384ce1ffe00000000061ffe00000000070e670c200001f807
     f00810aad4530d568a20600001f807f00810d568a30aad452070e670c21819
 (5) 000000000000000001e001e07f80ff79e001e07f80ff780000001e001e07f8
     0ff7800000000000001e1e001e07f80ff7800000000000001e00000001e1e0 [1]:1
  5. 7321c79e3c79e384ce1ffe00000000061f801f800000830aad452007e00007
     f00050d568a30e670c20600001f807f00810e670c30d568a20b0aad4521419
 (6) 000000000000000001e001e07f80ff79e001e07f80ff780000001e001e07f8
     0ff7800000000000001e1e001e07f80ff7800000000000001e00000001e1e0 [1]:1
  6. 7321c79e3c79e384ce1ffe00000000061f801f800000830aad452007e00007
     f00050e670c30d568a20600001f807f00810d568a30e670c20b0aad4521419
 (7) 000000000000000001e001e07f80ff79e001e07f80ff780000001e001e07f8
     0ff7800000000000001e1e001e07f80ff7800000000000001e00000001e1e0 [1]:1
  7. 7321c79e3c79e384ce1f801f800000821f801f800000830aad452007e00007
     f00050d568a30e670c20a007e00007f00050e670c30d568a20b0aad4521415
 (8) 000000000000000001e001e07f80ff79e001e07f80ff780000001e001e07f8
     0ff7800000000000001e1e001e07f80ff7800000000000001e00000001e1e0 [1]:1
  8. 7321c79e3c79e384ce1f801f800000821f801f800000830e670c2007e00007
     f0005099cc330e670c20a007e00007f00050e670c3099cc320b0e670c21415
 (9) 00c00861c386100300003c00199e3300003c00199e3300f180c00708190800
     000e200603ccf180c0194708190800000e2f180c0000603cd94f180c000000 [15]:4
(10) 000000000000000001e001e07f80ff79e001e07f80ff780000001e001e07f8
     0ff7800000000000001e1e001e07f80ff7800000000000001e00000001e1e0 [1]:1
  9. 7321c79e3c79e384ce1f801f800000821f801f800000830e670c2007e00007
     f00050aad4530d568a20a007e00007f00050d568a30aad4520b0e670c21415
(11) 000000000000000001e001e07f80ff79e001e07f80ff780000001e001e07f8
     0ff7800000000000001e1e001e07f80ff7800000000000001e00000001e1e0 [1]:1
 10. 7321c79e3c79e384ce1f801f800000821f8000007f00030aad452007e1f800
     000850d568a30e670c20a007e00007f00050e670c30d568a20d0aad4521215
(12) 000000000000000001e001e07f80ff79e001e07f80ff780000001e001e07f8
     0ff7800000000000001e1e001e07f80ff7800000000000001e00000001e1e0 [1]:1
 11. 7321c79e3c79e384ce1f801f800000821f8000007f00030aad452007e1f800
     000850e670c30d568a20a007e00007f00050d568a30e670c20d0aad4521215
(13) 000000000000000001e001e07f80ff79e001e07f80ff780000001e001e07f8
     0ff7800000000000001e1e001e07f80ff7800000000000001e00000001e1e0 [1]:1
 12. 7321c79e3c79e384ce1f801f800000821f8000007f00030e670c2007e1f800
     000850e670c3099cc320a007e00007f0005099cc330e670c20d0e670c21215
(14) 000000000000000001e001e07f80ff79e001e07f80ff780000001e001e07f8
     0ff7800000000000001e1e001e07f80ff7800000000000001e00000001e1e0 [1]:1
 13. 7321c79e3c79e3fc001f801099c000821f801f800000e10e641c2007e00006
     13385f800c330e6700782007e00007f00610e642c3099cc300f0e6700b8416
(15) 0000000000000003fde001e00000ff79e001e07f80ff000003c01e001e07f8
     0007807f800000000f001e001e07f80ff000003c000000001e000000f001e0 [3]:1
 14. 7321c79e3c79e3fc001f801099c000821f8000007f00610e641c2007e1f800
     000e10e642c3099cc300e007e0000613385f800c330e67007850e6700b8216
(16) 0000000000000003fde001e00000ff79e001e07f80ff000003c01e001e07f8
     0ff000003c000000001e1e001e07f80007807f800000000f00000000f001e0 [3]:1
 15. 73218f9f18f9f184ce8ffc400000000a8ffc400000000b38670c200001ce07
     9c081099ca2b38670c20600001ce079c08138670c3099ca2a0738670c22829
 (9) 00007060e7060e000000000c7998f30000000c7998f300c798000e0c1c0000
     0006a00035d4c798001a8e0c1c00000006ac79800000035d5a8c7980000000 [15]:2
(17) 0c1840602406021830600190806100e0600190806100e0c600a280f3c00198
     033003060514c600a28000f3c0019803300c600a28306051400c600a298786 [2]:1
(18) 0000000000000000017003a31f863f757003a31f863f7400000017003a31f8
     63f7400000000000001d17003a31f863f7400000000000001d00000001d1d0 [3]:2
 16. 7312cb5e3c7ad348ce1ffe00000000061f801f800000830aad452007e00007
     f00050d568a30e670c20600001f807f00810e670c30d568a20b0aad4521419
(19) 000000000000000001e001e07f80ff79e001e07f80ff780000001e001e07f8
     0ff7800000000000001e1e001e07f80ff7800000000000001e00000001e1e0 [1]:1
 17. 7312cb5e3c7ad348ce1ffe00000000061f801f800000830aad452007e00007
     f00050e670c30d568a20600001f807f00810d568a30e670c20b0aad4521419
(20) 000000000000000001e001e07f80ff79e001e07f80ff780000001e001e07f8
     0ff7800000000000001e1e001e07f80ff7800000000000001e00000001e1e0 [1]:1
 18. 7312cb5e3c7ad348ce1ffe00000000061f801f800000830b35862007e00007
     f00050e670c30cce4920600001f807f00810cce4930e670c20b0b358621419
(21) 000000000000000001e001e07f80ff79e001e07f80ff780000001e001e07f8
     0ff7800000000000001e1e001e07f80ff7800000000000001e00000001e1e0 [1]:1
 19. 7312cb5e3c7ad348ce1ffe00000000061f8000007f00030b35862007e1f800
     000850e670c30cce4920600001f807f00810cce4930e670c20d0b358621219
(22) 000000000000000001e001e07f80ff79e001e07f80ff780000001e001e07f8
     0ff7800000000000001e1e001e07f80ff7800000000000001e00000001e1e0 [1]:1
 20. 7312cb5e3c7ad348ce1f8000007f00021f8000007f00030e670c2007e1f800
     000850aad4530d568a20c007e1f800000850d568a30aad4520d0e670c21213
(23) 000000000000000001e001e07f80ff79e001e07f80ff780000001e001e07f8
     0ff7800000000000001e1e001e07f80ff7800000000000001e00000001e1e0 [1]:1
 21. 7312cb5e3c7ad348ce1f8000007f00021f8000007f00030e670c2007e1f800
     000850b358630cce4920c007e1f800000850cce4930b358620d0e670c21213
(24) 000000000000000001e001e07f80ff79e001e07f80ff780000001e001e07f8
     0ff7800000000000001e1e001e07f80ff7800000000000001e00000001e1e0 [1]:1
 22. 7301c79e3c79e3fc021fc00099c000c01f809f802000410e45146107e00004
     133a1f800c330c6700f82007f10003f00210e243c2099cc330b0a6628b8494
 23. 7301c79e3c79e3000205d49a987530041fc01f800000e10e26086107e00007
     f00610a6608a099cc328c1a2b1581ab03820000c330e451438b0c652439512
 24. 7301c79e3c79e3000205d49a987530041fc000007f00610e26086107e1f800
     000e10a6608a099cc328c1a2b1581ab03820000c330e451438d0c652439312
 25. 7301c79e3c79e300021fc01f800000e01f811099ff00030e241c2107e00007
     f0060099cc330a642c30c007e9f802133850e470060000c338b0c6700a9394
 26. 7301c79e3c79e300021fc01f800000e01f811f802133830e27006107e00007
     f0060099cc330a6700b0a007e9099ff00050e441c20000c338b0c642c29594
 27. 733fc79e0c79e0fc0010fe00180030641f81909980000308643ca00420001f
     f0361f99c0010e043c78400019f80213381086720ff800c33830e0710f9188
(25) 00000000f0000f03fde001e007800f01e000000000ff780783c01e001e0780
     0f0000003cf0007bc0001e000007f80007807800f007f80000000780f00000 [2]:1
 28. 733fc79e0c79e0fc0010801f980030e01f81909980000308643ca00420001f
     f0361f99c0010e043c788007f8000213305086720ff800c33830e0710f9184
(26) 00000000f0000f03fde001e007800f01e000000000ff780783c01e001e0780
     0f0000003cf0007bc0001e000007f80007807800f007f80000000780f00000 [2]:1
 29. 733fc79e0c79e0fc0010801f980030e01f818000213303086720e00420001f
     f0361f99c0010e0710f88007f909980000508643cbf800c33830e043c79184
(27) 00000000f0000f03fde001e007800f01e000007f80007807800f1e001e0780
     0f0000003cf000780f001e000000000ff780783c0007f800000007bc000000 [2]:1
 30. 733fc79e0c79e0000010fe1f980030641f818000213307086700e00421f81b
     f03e0019c0010e071079200019099ff000308642cbf800c33850e043c39988
(28) 00000000f0000f03fde001e007800f01e000007f80007807800f1e001e0780
     0f0000003cf000780f001e000000000ff780783c0007f800000007bc000000 [2]:1
 31. 733fc79e0c79e0000010801f980030e41f819581b53003086518a00420001f
     f0363f99c0010e0518b88007f9a982b030508662460000c33930e06247918c
 32. 733fc79e0c79e0000010801f980030e41f819f80213303086700e00420001f
     f0363f99c0010e0700f88007f9099bf000508643c20000c33930e043c3918c
(29) 00000000f0000f03fde001e007800f01e000007f80007807800f1e001e0780
     0f0000003cf000780f001e000000000ff780783c0007f800000007bc000000 [2]:1
 33. 733fc79e0c79e0000000c01f9840306400c000183f30e5f99c00015179a81b
     503030a2608f0a26383881a2b9581ab03030c4504f0c45343880000c339592
 34. 733fc79e0c79e0000000c01f9840306400c000183f30e5f99c000152b9581b
     503030a4508f0a45383881a179a81ab03030c2604f0c26343880000c339592
 35. 733fc79e0c79e0000000c01f9840306405e995983530050a4508e10020001b
     f03e3f99c0010a45383881a179a81ab03030c2604e0000c338d0c263439192
 36. 733fc79e0c79e0000000c01f980030e400c000187f3065f99c000152b9581b
     503030a452870a4518b881a179a81ab03030c2614b0c26247880000c339394
 37. 733fc79e0c79e0000000c01f980030e405e995983530050a4528610020001f
     f0363f99c0010a4518b881a179a81ab03030c2614a0000c338b0c262479194
 38. 733fc79e0c79e0000000c01f980030e41f819581b530030c2518a10020001f
     f0363f99c0010a4518b88007f9a982b03050a462460000c338b0c262479194
 39. 733fc79e08f98cac0c8f8c1a9800000210801f980110a301662ca00420001f
     9c2a3380a2c7f99c00188007f80003530095818a33086728a8b0e0718a8d90
 40. 733fc79e08f98cac0c8f8c000035300210801f980110a3016728a00420001f
     9c2a3380b18bf99c00188007f9a980000095818a3308662ca8b0e062c68d90
 41. 733f8f98c863ec84cc833e1f800000088f8c00003f0009099892a108000887
     92c8738830cb01672c21c00421ce141018706170c7380b1c21bf99c0019188
 42. 733f8f98c863ec84cc833e00003f00088f8c1f80000009099892a108000887
     92c8738832c301670ca1c00421ce141018706171c3380b0c61bf99c0019188
 43. 733f8f98c863ecac0c833e1a980000088f8c0000353009581892a108000887
     92ca338822cb016728a8c00421ce14101a306162c7380b18a8bf99c0019188
 44. 733f8f98c863ecac0c833e00003530088f8c1a98000009581892a108000887
     92ca3388328b01662ca8c00421ce14101a3061718b380a2c68bf99c0019188
 45. 733f8f98c863ecfc00833e00002133088f8c1099800009f80092a108000887
     92ce1388320f01643cb8400421ce14101e1061710f38083c783f99c0019188
 46. 733f8f98c861f084cc837d9f800000088f8c00003f000b099882a108000887
     9200738e70cb01672c21c00421ce160c18608170c7380b1c01ff99c0c80188
 47. 733f8f98c861f084cc837d80003f00088f8c1f8000000b099882a108000887
     9200738e72c301670ca1c00421ce160c18608171c3380b0c41ff99c0c80188
 48. 733f8f98c861f0ac0c837d9a980000088f8c000035300b581882a108000887
     9202338e62cb016728a8c00421ce160c1a208162c7380b1888ff99c0c80188
 49. 733f8f98c861f0ac0c837d80003530088f8c1a9800000b581882a108000887
     9202338e728b01662ca8c00421ce160c1a2081718b380a2c48ff99c0c80188
 50. 733f8f98c861f0fc00837d80002133088f8c109980000bf80082a108000887
     9206138e720f01643cb8400421ce160c1e0081710f38083c587f99c0c80188
 51. 733ec198c8f9807c008f8c1f00213318108e3f88411001016700e007f90896
     0cc0e0800a07c900ac9920cc2400179c203380b10d010f2cd89380c5c18784
 52. 733ec198c863e07c00833e1f00213318108e2008792c05388320c1f8190896
     0cc0e080090bc904ac1860cc25f81410101061700f010f1cd9138085c58b84
 53. 733e8fd8c863e00300830190897f0c0e833e000021330df800922108e3c807
     92c640184a01388300f940cc25f8181010138045c5061710f83010f0cd8d90
 54. 733e8fd8c861f00000837c006021330c830190897f000bf860b22108e3c807
     9206138e320e0184aca000cc25f81a0c1e2081000f38045c18b010f0c9998c
 55. 72388fb8c863fcfc3083c20ce14001808301800621318bf8141248f8e10988
     000093008bc7f80092b8210b24088392061309c2c9388720b8301643cb8984
 56. 72388fb8c863fc003083c20ce17f018083019f8621318a00141248f8e10988
     0000b3008bc3f80082b8a10b24088392065309c2c1388700b9301643cb8984
 57. 72388fb8c8600cfc3083c20ce1400182837f800621318bf8140248f8e1098b
     f00093008bc4000092b8610b240883920e1309c2c9388320b8301643c38994
 58. 72388fbfc863fce03083c20060c001e283018c80213191f8140268f8010988
     70009f810b03306091f8810b24088392061309c2c9388720b8301643cb8984
 59. 62118e3f18e29103008afe34d000ca8094f0438611228f356001c83803c667
     f006110a4a46380401292950088a0ffc08c3800a0330645438528549a02d43
 60. 62118e3f18e29103008a8034d03fca8094f0438611228f356401483fe3c664
     0006110a4a46380001a92950088a0ffc08c3800b0130645438528548a22d43
 61. 62118e3f18e29103008a8034d03fca8094f0438611228f356401483803c667
     f006110a4b44380401292957e88a0c0c08c3800a03306054b8528548a22d43
 62. 62118e3f18a78100308afe34d000ac809c5243e011228f365001c83803c667
     f006110a4a4a38040128a950088a0ff0c8a3800a0336046438528549a02d43
 63. 62118e3f18a78100308a8034d03fac809c5243e011228f365401483fe3c664
     0006110a4a4a380001a8a950088a0ff0c8a3800b0136046438528548a22d43
 64. 62118e2918fc0164009f80956071998094f0489147000b3061192807e2ae56
     844923800025356464b82807e336038cce0291522438646024d180242a3601
 65. 62118e2918fc0164009f80336038cce094f0489147000a3864502807e95607
     19980291522530612924c807e2ae56844931802427356464a0038000283783
 66. 62118e2918fd9100009faa20656842a294f0489147000b356074a804c3fe03
     8cce1184802638646020080549fe07190c838000233065093822924a283a4d
 67. 62118e2918fd9100009fcc206038cce094f0489147000a3860702802a3fe56
     842a31848027356464a0080549fe07190c82924a2330650924c38000283b83
 68. 62018e18386193800095fe20a01530449fea25e1a0c8a3011538a80d49ac07
     f008932662c23000cc2049a411fe66b048a306003238400440a0c24a22787a
 69. 62388e38c8e3fcff30908e2c901920008380006621110f30e422c83b390894
     000093084acbf800b23828cc2438080c181310f309f81012d8930145c78582
 70. 62388e38c8e3fcff30908e2c901920008381800621118b308722a83b210890
     0c00d30e4b09f80092b848cc243e0c00101310c2cdf8141258330107c78982
 71. 62388e38c8e3fc0330908e2c9019200483801fe621110f30e402c83b390897
     f00093084ac20000b23928cc2438080c183310f309f81002d8930145c3858a
 72. 62388e38c8e3fc0330908e2c9019200483819f8621118b308702a83b210893
     fc00d30e4b00000092b948cc243e0c00103310c2cdf8140258330107c3898a
 73. 62388e3fc8ff8ce330908e2c90001c80803203e021120df90022c83819088d
     900113084acb3060727828cc24000870163310f305361401f91f8149038588
 74. 62388e3fc8ff8c643090b24388001000800f9c60211191f90722a83c22060d
     9020330942cd300f027828c801090a6cc0b31647ca3800a1e01f810b01898c
 75. 3301823cfd29824c869709ff8020000812fa1cd6228c060d0304a17d200829
     1a2073e2b087019808c03c32830003f000c990c82618a92c01f86221c21c9c
 76. 3301823cfd29824c869709f0003f000812fa1cd6228c060d012c217d200829
     1a2073e2a1c3019808c05c3283f8020000c990c82618ab0481f86230861c9a
 77. 3301823cfd2982fc009709e990e0000812fa1cd6228c600d0034a17d200829
     1a2613e28387019808c03c328300013218df80082618a920d87862210f849c
 78. 485264169968264a1300f3fe56080501a1e1051618a8175b012189383e2502
     ec501a848985000128627d122832294c10f0001287a848984335b01218464d
 79. 485264169968264a1300f3fe56080501a1e1051618a8175b012189383e2502
     ec5010001287a84898427d122832294c10fa8489850001286335b01218464d
 80. 4912341699682c489300f9fcd6080501a1b0514618a2475b012189389e0d02
     ec501a81a185000128627d12881a294c10f0001287a81a184335b01218464d
 81. 4912341699682c489300f9fcd6080501a1b0514618a247a81a1849389e0d02
     ec5015b01219000128627d12881a294c10f00012875b0121833a81a184464d
 82. 4912341699682c489300f9fcd6080501a1b0514618a247a81a1849389e0d02
     ec50100012875b0121827d12881a294c10f5b01219000128633a81a184464d
 83. 4982341699682c419300fbbe96080501a1b2130618a2475b01418938ba2902
     ec50185161850001486179921813256810f00014878516184335b01418464d
 84. 4982341699682c419300fbbe96080501a1b2130618a2478516184938ba2902
     ec5015b014190001486179921813256810f00014875b01418338516184464d
 85. 4982341699690c099300fbbe96080029a1b2130618a2c35b01418938ba2902
     ec02985161830001486179905883256810f00280878512384335b280184a4d
 86. 4990309699690c099300fb1ebe080029a1961a0618a2c38512382938b02b82
     ec0295b2801900280861b9905883256818b00280875b280183385123824a4d
 87. 4990309699690c099300fb1ebe080029a1961a0618a2c38512382938b02b82
     ec02900280875b280181b9905883256818b5b2801900280863385123824a4d
 88. 405a00024240005a03a1915de5fc2a61a1915de5fc2a610040045dbdd88e08
     4752c990460300400443bdbdd88e084752d004004499046023b00400444e4f
(30) 2fa17480bd012e85f4000000000000000000000000000066bb9b8000000000
     0000066bb9b866bb9b8000000000000000066bb9b866bb9b80066bb9b80000 [2]:1
 89. 405a00024240005a03a1915f35f98a61a1915f35f98a610040045dbdd8a308
     1d52c990460300400443bdbdd8a3081d52d004004499046023b00400444e4f
(31) 2fa17480bd012e85f4000000000000000000000000000066bb9b8000000000
     0000066bb9b866bb9b8000000000000000066bb9b866bb9b80066bb9b80000 [2]:1
 90. 4186345689690c099308fabe86480029a1b2130618a2c35b01418d383a210a
     ec02905568830001486179905883256810f00280878512384335b280184a4d
 91. 64da6010c308065b2704921e843d436097df828405430110a545588050d907
     6c1a5855088705016001d8bf5010e500227050160185508861710a54543847
 92. 64da6010c308065b2704921e843d436097df828405430110a545588050d907
     6c1a5050160185508861d8bf5010e500227855088705016001710a54543847
 93. 64da6010c228145b2720905e843d436097df828407028102a945588050d907
     6c1a5855088705016001d8bf5010e500227050160185508861710a54543847
 94. 64da6010c068305b2720905e8438e36097df828402a28102a944dc3f1810e5
     002272801601a8508862788050d9076c1a9855088705016001d10a52543847
 95. 64da44128148225b2704921c5438e36097df805400e301a850887c0018d907
     6c1a910a524d28016002dc3f1810e50022b280160110a524c27a8508863847
 96. 402d3647c3647ccb31dead6044008823ad5ee022004445680400380000990f
     2cc015090d47a090e8a63800019661321815091137a09223a65966030e8911
(32) 000009b8309b83000000000ab2d56580000005d4cba98000033cc137060000
     00078ab2c0005d4c0019813706000000078ab2c0005d4c001980003cf00000 [2]:1
 97. 400024168168240003a4df9d45103ed1a4df9d45103ed10001083de8b8169f
     65428990c109000108237de8b8169f654290001082990c1083700010824e4f
(33) 2dd2d0217e840b4bb4000000000000000000000000000066f2e74000000000
     0000066f2e7466f2e74000000000000000066f2e7466f2e740066f2e740000 [2]:1
 98. 401e3647c3647ccb31decb6044008823b4dee022004444f00400380000990f
     2cc015090d47a090e8a63800019661321815091137a09223a65966030e8911
(34) 000009b8309b83000000000ab2d56580000005d4cba98000033cc137060000
     00078ab2c0005d4c0019813706000000078ab2c0005d4c001980003cf00000 [2]:1
 99. 401234d20349304803df7d801e00c367bf7d801e00c367680000180019c967
     9c2e0408c30f8888000678001938e67c1e11910000211430e6768032c560e1
(35) 3300c000fc000f03fc008200800100000082008001000007fbcf0030c03000
     6000007738e007738e0000c300c0018000006eb4d006eb4d0000198c301e1e [3]:1
100. 401234d20349304803df7d801e00c367bf7d801e00c367680000180019c967
     9c2e14888000808c30e678001938e67c1e0111430f29100006768032c560e1
(36) 3300c000fc000f03fc008200800100000082008001000007fbcf0030c03000
     6000007738e007738e0000c300c0018000006eb4d006eb4d0000198c301e1e [3]:1
101. 4012049200492048039f807fe600cc9b9f807fe600cc9b6800001807e60667
     fcc1c900430f68000006b807e60667fcc1d6800000900430e6b68000007475
(37) 00c00b6dc3b6d0030000240019d2330400240019d233049184c02768196800
     0006000600c09184c0380768196800000609184c0200600c1809184c020a0a [16]:4
(38) 3fc0f000ff000f03fc000000000000000000000000000007fbcf0000000000
     0000007fbcf007fbcf0000000000000000007fbcf007fbcf00007fbcf00000 [1]:1
102. 4012049200492048039f807fe600cc9b9f807fe600cc9b6800001807e60667
     fcc1c600430f98000006b807e60667fcc1d9800000600430e6b68000007475
(39) 3fc0f000ff000f03fc000000000000000000000000000007fbcf0000000000
     0000007fbcf007fbcf0000000000000000007fbcf007fbcf00007fbcf00000 [1]:1
103. 403f0492c0492cff31897e6006000c059fa4606000c003f8073cd800009007
     fcc196980f096818f040580001fe6120099698023d681813c03fe6030e1008
(40) 00003b6d33b6d3000000000961d2c3800000091e523c800000c3076da60000
     00060961c00291e40026076da6000000060961c00291e40026000030c00606 [2]:3
(37) 0000c0000c000000ce76818918523082605b8901d20384000000216db96800
     03380006300406030098216db8019ed000000601c006002c184019cc300e17 [16]:2
\end{verbatim}

\section*{Acknowledgments}
Obtaining the results of Sections~\ref{s:class} and~\ref{s:prop}
was funded by the Russian Science Foundation (grant 18-11-00136);
the work on Section~\ref{s:related}
was supported within the framework 
of the state contract of the Sobolev Institute
of Mathematics (FWNF-2022-0017).
The author thanks the Supercomputing
Center of the Novosibirsk State University for provided computational resources.


\providecommand\href[2]{#2} \providecommand\url[1]{\href{#1}{#1}}
  \def\DOI#1{{\small\url{https://doi.org/#1}}}

\end{document}